\documentclass[a4paper,reqno,twoside]{amsart}
\pdfoutput=1

\usepackage{amsthm}
\usepackage{amsmath}
\usepackage{amsfonts}
\usepackage{stmaryrd}
\usepackage{amssymb}
\usepackage{mathrsfs}

\usepackage{hyperref}

\usepackage[symbol]{footmisc}
\usepackage{microtype}

\usepackage{todonotes}
\usepackage{setspace}
\usepackage{tikz,tikz-cd}
\usetikzlibrary{arrows,%
                shapes,positioning}

\DeclareMathOperator{\PSL}{PSL}
\DeclareMathOperator{\PSO}{PSO}
\DeclareMathOperator{\SO}{SO}
\DeclareMathOperator{\PO}{PO}
\DeclareMathOperator{\GL}{GL}

\newcommand{\SOAlg}{\mathbf{SO}}
\DeclareMathOperator{\Orth}{O}

\newcommand{\Hyp}{\mathbb{H}}
\newcommand{\G}{\mathbf{G}}
\newcommand{\Res}{\mathbf{Res}}
\DeclareMathOperator{\Gal}{Gal}

\newcommand{\gps}{\textsc{gps}}

\newcommand{\R}{\mathbb{R}}
\newcommand{\Q}{\mathbb{Q}}
\newcommand{\Z}{\mathbb{Z}}
\newcommand{\N}{\mathbb{N}}
\newcommand{\Int}{\mathcal{O}}
\DeclareMathOperator{\Isom}{Isom}
\newcommand{\FD}{\mathcal{F}}
\newcommand{\Gbar}{\overline{G}}

\newcommand{\id}{\mathrm{id}}
\newcommand{\Zclosure}[2]{\overline{#1}^{(#2)}}

\renewcommand{\epsilon}{\varepsilon}
\renewcommand{\setminus}{\smallsetminus}
\renewcommand{\leq}{\leqslant}
\renewcommand{\geq}{\geqslant}

\theoremstyle{theorem}
\newtheorem{theorem}{Theorem}[section]
\newtheorem{theorem*}[theorem]{Theorem}
\newtheorem{proposition}[theorem]{Proposition}

\newtheorem{corollary}[theorem]{Corollary}
\theoremstyle{definition}
\newtheorem{definition}[theorem]{Definition}
\theoremstyle{remark}
\newtheorem*{remark}{Remark}
\newtheorem*{claim}{Claim}

\title{Quasi-arithmeticity of lattices in $\PO(n,1)$}
\author{Scott Thomson}
\date{20th May 2015}

\address{Max Planck Institut f\"ur Mathematik, Vivatsgasse 7, 53111 BONN, Germany}
\email{sathomson@mpim-bonn.mpg.de}
\email{s.a.thomson@dunelm.org.uk}
\thanks{The author was supported during the earlier part of this work by Swiss NSF grant \texttt{200021\_144438}, and is grateful to the Max Planck Institute for Mathematics (Bonn) for its hospitality and financial support during the preparation of the manuscript.}

\keywords{quasi-arithmetic group, arithmetic group, lattice, hyperbolic manifold}
\subjclass[2010]{22E40 (primary), 20H10, 11F06 (secondary).}

\begin{document}

\maketitle

\begin{abstract}
	We show that the non-arithmetic lattices in $\PO(n,1)$ of Belolipetsky and Thomson \cite{BT:Systoles}, obtained as fundamental groups of closed hyperbolic manifolds with short systole, are quasi-arithmetic in the sense of Vinberg, and, by contrast, the well-known non-arithmetic lattices of Gromov and Piatetski-Shapiro are \emph{not} quasi-arithmetic. A corollary of this is that there are, for all $n\geq 2$, non-arithmetic lattices in $\PO(n,1)$ that are not commensurable with the Gromov--Piatetski-Shapiro lattices.
\end{abstract}

\setcounter{tocdepth}{1}

\tableofcontents

\section{Introduction}

\subsection{Background, motivation and discussion}\label{sec:discussion}

The reader will find precise statements of theorems in \S\ref{sec:summary} (p.\pageref{sec:summary}), following the definitions given in \S\ref{sec:defs}.
Beforehand, we give some background and context.

\par The study of locally symmetric spaces is often re-framed as the study of discrete subgroups of semisimple Lie groups and in particular of those groups acting with finite co-volume (which are usually called `lattices').
A well known source of examples of lattices in semisimple Lie groups are the `arithmetic' lattices; i.e., subgroups of algebraic groups (over $\Q$) that are defined over $\Z$.
That the arithmetic subgroups of semisimple Lie groups have finite co-volume was shown by A.~Borel and Harish-Chandra \cite{Borel-Harish-Chandra}, who also obtained co-compactness criteria.
(The co-compactness criteria were also obtained by G.~Mostow and T.~Tamagawa \cite{Mostow-Tamagawa}.)
If the real rank of a semisimple Lie group $G$ is at least~$2$, then by results of G.~Margulis it is known that any lattice in $G$ is arithmetic \cite[(A), p.298]{Margulis}.
For the case of real rank~$1$ Lie groups, and in particular for $\PO(n,1)$ (the isometry group of real hyperbolic $n$-space), it was not until the work of M.~Gromov and I.~Piatetski-Shapiro \cite{GPS} that one knew that there are non-arithmetic lattices in $\PO(n,1)$, for every $n\geq 2$, alongside the previously known arithmetic ones.
The examples of Gromov and Piatetski-Shapiro arise as fundamental groups of finite-volume hyperbolic manifolds, constructed by `gluing' together pieces of non-commensurable arithmetic hyperbolic manifolds along isometric totally geodesic boundaries.
Before this point, examples of non-arithmetic lattices in $\PO(n,1)$ were known for some small $n$, and these arise as reflection groups whose non-arithmeticity may be deduced from criteria determined by \`{E}.~Vinberg \cite{Vinberg:DiscreteGroupsReflections} \cite[3.2, p.227]{VinbergShvartsman:Discrete}.

\par In his work on the arithmeticity of these reflection groups, Vinberg introduced the class of `quasi-arithmetic' lattices (in a given Lie group) \cite[p.437]{Vinberg:DiscreteGroupsReflections}.
(Note that `quasi-arithmeticity' is different from `sub-arithmeticity': cf.~\S\ref{sec:Sub-Arith}.)
Every arithmetic group is quasi-arithmetic, whilst on the other hand Vinberg himself gave examples of lattices that are quasi-arithmetic but \emph{not} arithmetic (lattices that we will call `properly quasi-arithmetic').
However, Vinberg's examples are only given for dimensions no greater than $4$, and it appears that his definition has not since been considered a great deal in the literature, save once, to the author's knowledge \cite{HLM:Borromean}.
Vinberg has, however, produced more examples of quasi-arithmetic groups in dimension $2$ (though without explicitly using this terminology) \cite{Vinberg:SomeExamples}.
Vinberg computes the rings of definition of these groups and infinitely many of these are not equal to $\Z$ and so hence infinitely many of his examples are non-arithmetic.

\par M.~Norfleet has also recently produced a family of examples of Fuchsian groups, which, being contained in $\PSL_{2}(\Q)$, are quasi-arithmetic \cite{Norfleet:ManyExamples}.

\par It was recently shown by M.~Belolipetsky and S.~Thomson \cite{BT:Systoles} that one may obtain a class of non-arithmetic lattices in $\PO(n,1)$, by a construction of closed hyperbolic manifolds with short closed geodesics:
\begin{theorem*}[B.--T.\ \cite{BT:Systoles}]
Let $\epsilon>0$ and let $n\geq 2$.
Then there exists a closed hyperbolic $n$-manifold $M$ such that $M$ contains a non-contractible closed geodesic of length less than $\epsilon$.
Moreover, for $\epsilon$ small enough (depending on $n$ and some parameters in the construction of $M$), $M$ is a non-arithmetic hyperbolic manifold.
\end{theorem*}
\noindent This construction involves `gluing' in the spirit of Gromov and Piatetski-Shapiro, but in these more recent examples the manifolds being glued together are still commensurable with one another and so one might expect that the `non-arithmeticity' that arises should be `weaker' than that of Gromov and Piatetski-Shapiro.
Here, we show that the non-arithmetic examples of Belolipetsky and Thomson are, in fact, properly quasi-arithmetic lattices (\S\ref{sec:QAShortSystole}).

\par In order to establish a distinction between the two above classes of examples of non-arithmetic lattices, we also show here that the lattices of Gromov and Piatetski-Shapiro are \emph{not} quasi-arithmetic (\S\ref{sec:GPSnonQuasiArith}).

\subsubsection*{Acknowledgements}
The author is grateful to V.~Emery for a question that led to the writing of this article, and to J.~Raimbault for some helpful comments, as well as to the Max Planck Institut f\"{u}r Mathematik for its hospitality and financial support.
The author was supported during the earlier part of this work by Swiss NSF grant \texttt{200021\_144438}.
Many thanks also go to an anonymous referee who offered several useful comments and references.

\subsection{Definitions and foundational material}\label{sec:defs}
Here we give a definition of quasi-arithmetic lattices and establish some notation for the rest of the article.
For a general reference on the theory of arithmetic groups the reader may wish to consult the books of Margulis \cite{Margulis} and Platonov and Rapinchuk \cite{PlatonovRapinchuk}; whereas for an introduction to the subject, the book of Morris \cite{WitteMorris} (for example) is directed towards the newcomer.

\begin{definition}\label{def:admissible}
	Let $\Gbar$ be a connected semisimple real Lie group without compact factors and with trivial centre.
	We say that an algebraic group $\G$ is \emph{admissible for $\Gbar$} if
	\begin{enumerate}
		\item	$\G$ is defined over $\Q$;
		\item	there exists a \emph{surjective} homomorphism $\phi\colon \G(\R)^{\circ} \to \Gbar$; and
		\item	the homomorphism $\phi$ has compact kernel.
	\end{enumerate}
\end{definition}
We think of the algebraic group $\G$ as a model for the Lie group $\Gbar$, whose algebraic structure allows us to obtain, easily, interesting classes of subgroups as described below.
From a geometric point of view, we are interested in these subgroups' images in $\Gbar$.

\par Recall that a discrete subgroup $\Gamma<\Gbar$ is a \emph{lattice} if the quotient $\Gamma\backslash\Gbar$ has finite volume induced by the Haar measure on $\Gbar$.

\newcounter{templistcounterST}
\begin{definition}\label{def:quasiarith}
	Suppose that $\Gbar$ is as in Definition~\ref{def:admissible} and that $\Gamma< \Gbar$ is a lattice.
	We say that $\Gamma$ is \emph{quasi-arithmetic} if
	\begin{enumerate}
		\item	there exists an admissible algebraic group $\G$ for $\Gbar$; and
	\item\label{2ndstatement}	there exists a finite-index subgroup $\Gamma'\leq \Gamma$ such that $\Gamma'\subseteq \phi\bigl(\G(\Q)\bigr) $.
	\end{enumerate}
	We will say that $\Gamma$ is \emph{arithmetic} if, in addition to being quasi-arithmetic, the following stronger statement holds:
	\begin{itemize} 
		\item\label{3rdstatement}	$\Gamma$ is commensurable with $\phi\bigl(\G(\Z)\bigr) $.
	\end{itemize}
	If $\Gamma$ is quasi-arithmetic but not arithmetic then it is called \emph{properly quasi-arithmetic}.
\end{definition}

\noindent Note that arithmeticity implies quasi-arithmeticity.
On the other hand Theorem~\ref{thm:quasiarithsystole} shows that, in $\PO(n,1)$ (for any $n\geq 2$), quasi-arithmeticity does not imply arithmeticity (cf.~\S\ref{sec:discussion}) and so the class of properly quasi-arithmetic lattices is non-empty in this context.

\begin{remark}
The term `quasi-arithmetic' should not be confused with the term `sub-arithmetic', introduced by Gromov and Piatetski-Shapiro.  (Cf.~\S\ref{sec:Sub-Arith}.)
\end{remark}

\subsubsection{Hyperbolic space}
We adopt the Lorentz model of hyperbolic $n$-space, where we first equip $\R^{n+1}$ with the quadratic form $f_{n}$,  given in the standard basis by
\[ f_{n}(x) = -x_{0}^{2} + x_{1}^{2} + \cdots + x_{n}^{2} , \]
and then denote
\[ \Hyp^{n} = \bigl\{ x\in \R^{n+1} \mid f_{n}(x) = -1 \text{ and } x_{0} > 0 \bigr\} \]
with metric $\mathrm{d}_{\Hyp^{n}}$ derived from the Lorentzian inner product $(\cdot,\cdot)$ associated to the quadratic form $f_{n}$ by
\[ \cosh\bigl( \mathrm{d}_{\Hyp^{n}}(x,y) \bigr) = - (x,y). \]
We then have $\Isom(\Hyp^{n}) \cong \PO(n,1) = \Orth(\R^{n+1},f_{n}) / \{\pm 1\}$ where the equality is by definition and where $\Orth(\R^{n+1},f_{n})$ is the orthogonal group $\{g\in\GL_{n}(\R) \mid f(g(x)) = f(x) \;\;\forall x\in \R^{n+1}\}$ \cite{Ratcliffe}.
The group of orientation-preserving isometries of $\Hyp^{n}$ (denoted $\Isom^{+}(\Hyp^{n})$) is isomorphic to $\PSO(n,1)$, which is in turn defined to be the quotient $\SO(\R^{n+1},f_{n}) / \{\pm 1 \}$: this has index $2$ in $\PO(n,1)$.
The groups $\PO(n,1)$ and $\PSO(n,1)$ are real Lie groups.
If $K$ is a totally real algebraic number field (with a fixed embedding $K\subset\R$), and if $f$ is any quadratic form over $K$ with signature $(n,1)$ then we may consider the algebraic group $\SOAlg_{f}$, and we have a Lie group isomorphism $\SOAlg_{f}(\R)^{\circ}\to\PSO(n,1)$.
We think of $\PSO(n,1)$ as a fixed concrete realisation of $\Isom^{+}(\Hyp^{n})$ and $\SOAlg_{f}(\R)$ as an algebraic `model' over $K$.

\par In what follows we will have $G=\Isom(\Hyp^{n})=\PO(n,1)$.
Then, the group $\Gbar = G^{\circ}$ will satisfy the hypothesis of Definition~\ref{def:admissible}.
We will not be considering non-orientable lattices here, though we will need to consider some isometries lying in $G \setminus G^{\circ}$.

\subsubsection{Standard arithmetic lattices}
As before, we still have $G=\Isom(\Hyp^{n})$ and $\Gbar=G^{\circ}$.
If $\G$ is some algebraic group over a number field $K$, then we can form a related algebraic group over $\Q$, denoted $\Res_{K/\Q}(\G)$, such that $\G(K) \cong \Res_{K/\Q}(\G)(\Q)$.
If $K$ is totally real, of degree $d$ over $\Q$, then we have $\G(K\otimes_{\Q}\R) \cong \Res_{K/\Q}(\G)(\R)$; moreover since $K\otimes_{\Q}\R \cong \R^{d}$ we obtain a direct product
\begin{equation}\label{eqn:restrscalars} \Res_{K/\Q}(\G)(\R) \cong \prod_{\sigma\in \Gal(K/\Q)} \G^{\sigma}(\R).
\end{equation}
The algebraic group $\Res_{K/\Q}(\G)$ is called the Weil restriction of $\G$, and we call $\Res_{K/\Q}$ the restriction of scalars functor \cite[p.50]{PlatonovRapinchuk} \cite[I.1.7]{Margulis}.

\begin{definition}
Fix some $n\geq 2$.
Let $K$ be a totally real number field, and let $f$ be a ($K$-valued) non-degenerate quadratic form on $K^{n+1}$.
Let $\SOAlg_{f}$ be the algebraic $K$-group determined by $ \{ x\in  \Orth(K^{n+1},f) \mid \det(x)=1 \}$.
We say that $(K,f)$ is an \emph{admissible field-form pair} if the group $\Res_{K/\Q}(\SOAlg_{f})$ is admissible for $\Gbar$.
Equivalently, by \eqref{eqn:restrscalars}, the pair $(K,f)$ is admissible if $\SOAlg_{f}(K\otimes_{\Q}\R)^{\circ}$ is isomorphic to $\Gbar$, modulo a compact kernel.
\end{definition}

\noindent (One could do away with reference to $K$ as this field is implicit in the definition of $f$.
We keep $K$ for emphasis on the field of definition, especially as we will later have occasion to consider extensions of scalars of the form $K\otimes L$. )

\par In more concrete terms, we may fix some basis for $K^{n+1}$, so realising $f$ as a homogeneous quadratic polynomial with coefficients in $K$.
We also regard $K$ as embedded in $\R$.
Then for $(K,f)$ to be admissible for $\PO(n,1)$ requires that $f$ has signature $(n,1)$ on $\R^{n+1}$ and that for all elements $\sigma\in\Gal(K/\Q) \setminus\{\id\}$ the conjugate forms $f^{\sigma}$ (obtained by applying $\sigma$ to the coefficients of $f$) are positive definite.

\par If $(K,f)$ is an admissible pair, and if $\Int_{K}$ denotes the ring of integers in $K$, then it is well-known that $\SOAlg_{f}(\Int_{K})$ can be realised as an arithmetic lattice in $\PO(n,1)$ \cite[(3.2.7)]{Margulis}.
This follows from \eqref{eqn:restrscalars}, since $\Res_{K/\Q}(\SOAlg_{f})(\Z)$ is a lattice in $\Res_{K/\Q}(\SOAlg_{f})(\R)$ and on projecting to the non-compact factor (which is the \emph{only} such factor by admissibility of $(K,f)$), we obtain a lattice in $\SOAlg_{f}(\R)$, isomorphic to $\SOAlg_{f}(\Int_{K})$.
To realise a lattice in $\PO(n,1)$, note that we have a homomorphism $\SOAlg_{f}(\R)^{\circ} \to \PO(n,1)$ with finite co-kernel.

\par We call $\SOAlg_{f}(\Int_{K})$ the \emph{standard arithmetic lattice} associated to the pair $(K,f)$.
Lattices of this type may also be called arithmetic lattices of the simplest type \cite[p.217]{VinbergShvartsman:Discrete}. 
By abuse of notation we will also refer to $\PO_{f}(\Int_{K})$ as the associated standard arithmetic lattice, when we wish to work with the image of this lattice in $\PO_{f}(\R)$.

\subsection{Summary and results}\label{sec:summary}
In \S\ref{sec:ABT}, we will recall the construction of I.~Agol, generalised by Belolipetsky and Thomson, of hyperbolic manifolds with short closed geodesics, and then show that their associated lattices are quasi-arithmetic.
This leads to the following theorem:

\begin{theorem}\label{thm:quasiarithsystole}
	For any admissible field-form pair $(K,f)$ there are infinitely many commensurability classes of \emph{properly} quasi-arithmetic lattices $\Gamma < \PO(n,1)$ arising from $\SOAlg_{f}(K)$.
\end{theorem}
\noindent (A proof of Theorem~\ref{thm:quasiarithsystole} is given on p.\pageref{pf:mainthmproof}.)

In \S\ref{sec:GPS} we will examine the construction of Gromov and Piatetski-Shapiro and show that the lattices so obtained are not quasi-arithmetic:

\begin{theorem}\label{thm:GPSnonQuasiArith}
	Let $\Gamma$ be a Gromov--Piatetski-Shapiro lattice.
Then $\Gamma$ is not quasi-arithmetic.
\end{theorem}

\noindent The definition and construction of a `Gromov--Piatetski-Shapiro lattice' is given in \S\ref{sec:GPSConstruction}.
Theorems~\ref{thm:quasiarithsystole} and \ref{thm:GPSnonQuasiArith} together give us the following:
\begin{corollary}
There are infinitely many commensurability classes of non-arithmetic lattices in $\PO(n,1)$ (for $n\geq 2$) that are not commensurable with Gromov--Piatetski-Shapiro lattices.
\end{corollary}

\section{Hyperbolic manifolds with short systole}

If $M$ is a non-simply-connected closed Riemannian manifold, then the \emph{systole} of $M$ is by definition the (length of) the shortest closed non-trivial curve in $M$.
One may show that the systole of such an $M$ is always positive \cite[p.39]{Katz:Systolic}.

\par In this section we will revisit a recent construction of closed hyperbolic $n$-manifolds with `short' systole, and then see that the corresponding lattices (i.e., fundamental groups) in $\PO(n,1)$ are quasi-arithmetic.

\subsection{Constructing hyperbolic manifolds with short systole}\label{sec:ABT}

Let $n\geq 2$ and $\epsilon > 0$.
We will describe how one can construct a closed hyperbolic $n$-manifold with systole at most $\epsilon$, as per the construction of Belolipetsky and Thomson \cite{BT:Systoles}.
This will allow us to establish some notation.
The method is a generalisation of a construction by I.~Agol in dimension~$4$ \cite{Agol:systoles}.
(It was also pointed out by N.~Bergeron, F.~Haglund and D.~Wise that some of their own methods allowed Agol's $4$-dimensional construction to be generalised to every dimension using subgroup separability arguments \cite[Remark on p.17]{BHW}.)

\par Fix some admissible pair $(K,f)$ with $K\neq \Q$, and let $\Lambda$ be a torsion-free subgroup of the associated standard arithmetic lattice in $\PO(n,1)$, so that $\Lambda\backslash\Hyp^{n}$ is a compact hyperbolic $n$-manifold $N$.
So, for some Galois embedding $\sigma\colon K\to \R$, the form $f^{\sigma}$ on $K^{n+1} \otimes_{\sigma(K)} \R$ has signature $(n,1)$ and we can isometrically identify $\Hyp^{n}$ with the more convenient model
\[ \Hyp_{f}^{n} = \bigl\{  x\in K^{n+1} \otimes_{\sigma(K)} \R \mid f(x) < 0 \bigr\} / \sim , \] 
where $x\sim \lambda x$ for all $x\in K^{n+1} \otimes_{\sigma(K)} \R$ and all $\lambda\in\R\setminus{0}$.
In this model, we have $\Isom(\Hyp^{n}_{f}) \cong \PO_{f}(K\otimes_{\sigma(K)}\R)$.
We will continue to refer to $\Hyp^{n}$ rather than $\Hyp^{n}_{f}$.
Let $H_{0}$ and $H_{1}$ be two disjoint `$K$-rational' hyperplanes in $\Hyp^{n}_{f}$; that is, choosing two vectors $v_{0}$ and $v_{1}$ in $K^{n+1}$ with $f(v_{i})>0$ for $i=0,1$, let $H_{i} = \langle v_{i}\rangle^{\perp_f} \cap \Hyp^{n}$, and suppose that $H_{0} \cap H_{1}=\emptyset$.
If one is interested in obtaining a manifold with short systole then one chooses the $v_{i}$ so that the $H_{i}$ are at most hyperbolic distance $\epsilon /2 $ apart.
The projections $\pi(H_{i}) \subseteq N$ are immersed totally geodesic hypersurfaces in $N$, but they might not be embedded and they might intersect eachother.
However, by replacing $\Lambda$ by a suitable finite-index (congruence) subgroup this can be avoided, so that $H_{i}$ ($i=0,1$ respectively) projects to a totally geodesic embedded hypersurface $N_{i}\subset N$ ($i=0,1$ respectively), and so that $N_{0}\cap N_{1} = \emptyset$ \cite[Lem.~3.1]{BT:Systoles}.

\par We now cut along the two hyperplanes $N_{i}$ to obtain a manifold with boundary.
Keeping the connected component $M'$ that contains the common perpendicular geodesic segment $c$ between the two $H_{i}$, we double along the boundary $B$ of $M'$ to obtain a manifold $M=DM'$ \cite[p.226]{Lee:SmoothMflds}.
Then $M$ contains the closed geodesic $Dc$ of length at most $\epsilon$. 
We will suppose that $B$ has $\ell$ connected components, and depending on whether or not the cutting separates the manifold $N$, we have $2\leq \ell \leq 4$.

\par Thus $M$ is a compact hyperbolic $n$-manifold that can be written as $\Gamma\backslash\Hyp^{n}$ for some lattice $\Gamma\in \PO_{f}(K\otimes_{\sigma(K)}\R)$.
If $\epsilon$ is small enough then the manifold $M$ (having systole at most $\epsilon$) is non-arithmetic \cite[5.1]{BT:Systoles}.

\subsection{Quasi-arithmeticity of short systole manifolds}\label{sec:QAShortSystole}

Let $M$ be as in \S\ref{sec:ABT}, with sufficiently short systole that it is non-arithmetic, and write $M=\Gamma\backslash\Hyp^{n}$.
We still suppose $(K,f)$ to be an admissible pair for $\PO(n,1)$.
We prove the following:

\begin{proposition}\label{prop:QuasiArith}
	The group $\Gamma$ can be generated by elements in $\PO_{f}(K)$ and so is quasi-arithmetic.	
\end{proposition}

\begin{proof} 
By construction, the manifold $M$ is a union $M_{1}\cup_{B} M_{2}$ (where $M_{i}$ is isometric to $M'$), so it has a symmetry $\tau$ that interchanges its two parts $M_{1}$ and $M_{2}$.
By this decomposition we see that the lattice $\Gamma \cong \pi_{1}(M)$ splits as the fundamental group of a graph of groups $\mathcal{G}$, where the graph $\mathcal{G}$ is the graph with two nodes and edges between the two nodes corresponding to each boundary component along which the double is taken.

That is, the edge groups are the fundamental groups of the boundary components and the vertex groups are the (isomorphic) groups $\Gamma_{1}$ and $\Gamma_{2}$.
(See Fig.~\ref{fig:GraphofGroups}.)

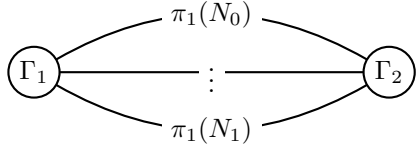
\begin{figure}\caption{The graph of groups $\mathcal{G}$ whose fundamental group is equal to that of the fundamental group of $M$.
		There are between $2$ and $4$ edges, one or two for $N_{0}$ and one or two for $N_{1}$.}\label{fig:GraphofGroups}
\begin{center}
\begin{tikzpicture}[node distance   = 4 cm]
  \useasboundingbox (-1,-1) rectangle (5,1.5); 
  \tikzset{VertexStyle/.style = {draw,
	  			thick,
	  			shape          = circle,
				fill     = white,
                                 text           = black,
                                 inner sep      = 2pt,
                                 outer sep      = 0pt,
                                 minimum size   = 18 pt}}
  \tikzset{EdgeStyle/.style   = {thick}}
  \tikzset{LabelStyle/.style =   {
	  thin,
                                  fill           = white,
                                  text           = black}}
	\node[VertexStyle](A){$\Gamma_{1}$};
	\node[VertexStyle,right=of A](B){$\Gamma_{2}$};
     \draw[EdgeStyle](A) to node[LabelStyle]{$\vdots$} (B);
     \tikzset{EdgeStyle/.append style = {bend left}}
     \draw[EdgeStyle](A) to node[LabelStyle]{$\pi_{1}(N_{0})$} (B);
     \draw[EdgeStyle](B) to node[LabelStyle]{$\pi_{1}(N_{1})$} (A);
  \end{tikzpicture}
\end{center}
\end{figure}

So, $\Gamma$ is generated by the elements of the two fundamental groups $\Gamma_{i} = \pi_{1}(M_{i})$ (for $i=1,2$), as well as some extra elements corresponding to the new homotopy classes that appear as a result of the gluing.
(See Serre~\cite{Serre:Trees} for the properties of graphs of groups.)

In what follows we describe how to concretely realise $\Gamma$ as a subgroup of $\PO_{f}(K)$.
\begin{claim}\label{claim:InjectivePi1}
	We may assume that $\Gamma_{1}$ can be identified with its image in $\Lambda$, on identifying $M_{1}$ with $M'\subset N$.
\end{claim}

\begin{proof}[Proof of claim]\label{rem:injectivepi1}
	\textit{A priori} it is not in general possible to make this latter identification; in particular if the hypersurfaces $N_{i}$ do not each separate $N$. 
But, there exists, nevertheless, a double cover $\pi\colon N'\to N$ such that each $\pi^{-1}(N_{i})$ \emph{does} separate $N$ \cite[2.8C]{GPS}.
Then, the piece obtained by cutting along both $\pi^{-1}(N_{i})$ is isometric to the original $M'$, and $\pi_{1}(M')\hookrightarrow \pi_{1}(N')$. 
Since $N'\to N$ is a finite cover, the discrete subgroup $\pi_{1}(N')<\pi_{1}(N)$ is still arithmetic.
(We may need to take a cover for each hypersurface $N_{i}$; but these then have a common cover, obtained by intersecting their fundamental groups, which is still a finite cover of $N$.)
\end{proof}

\par Now, write $I_{0}$ for the reflection in the hyperplane $H_{0}$: we have $\Gamma_{2} = I_{0}^{-1}\Gamma_{1}^{\vphantom{-1}}I_{0}^{\vphantom{-1}}$, and we note that $I_{0}\in\PO_{f}(K)$ since $H_{0}$ has a normal vector $v_{0}\in K^{n+1}$ (cf.~\S\ref{sec:ABT}).

The group $\Gamma_{i}$ is a discrete convex-co-compact group acting on $\Hyp^{n}$ (for both $i=1$ and $i=2$).
Choosing some basepoint $x_{0} \in H_{0}$ we may construct Dirichlet fundamental domains $\FD$, $\FD_{1}$ and $\FD_{2}$ (about $x_{0}$) for $\Gamma$, $\Gamma_{1}$ and $\Gamma_{2}$ respectively.
By the doubling construction it is evident that $\FD$ will be invariant under the reflection $I_{0}$.
Thus we can view $\FD$ as a union $F_{1}\cup F_{2}$ of the two pieces exchanged by $I_{0}$, and each $F_{i}$ satisfies the inclusion $F_{i}\subseteq \FD_{i}$.
We also have the inclusion $\FD\subseteq\FD_{i}$ ($i=1,2$) since $\Gamma_{i}<\Gamma$.
The intersection $\FD_{1}\cap \FD_{2}$ is a Dirichlet fundamental domain (at $x_{0}$) for the group $ \langle \Gamma_{1}, \Gamma_{2} \rangle$ and so we have $\FD \subseteq \FD_{1} \cap \FD_{2}$ since $\langle \Gamma_{1}, \Gamma_{2} \rangle \leq \Gamma$.

\par The domains $\FD_{i}$ may be decomposed into two parts separated by $H_{0}$.
By $\tilde{\FD_{i}}$ we mean the part containing the ends corresponding to the boundary components whose lifts to $\Hyp^{n}$ are $H_{1}$ and $\gamma_{j}(H_{j})$ ($j=0,1$) if either of the latter exist (cf.~\S\ref{sec:ABT}).
If $H_{0}$ intersects a bounding hyperplane of $\FD_{i}$ then it must do so orthogonally and hence we have the inclusion $I_{0}(\tilde{\FD_{i}}) \subseteq \FD_{i}\setminus\tilde{\FD_{i}}$.
Let the set $J$ be the index set $\{1, \ldots,\ell - 1\}$, and if $2\in J$ or $3\in J$ then denote the two hyperplanes $H_{2}=\gamma_{0}(H_{0})$ and $H_{3}=\gamma_{1}(H_{1})$.
Write $I_{j}$ for the reflection in $H_{j}$ (whenever $H_{j}$ exists), and write $H_{j}^{-}$ for the half-space bounded by $H_{j}$ and containing $H_{0}$.
The set $\FD$ is the intersection
\[ \FD = (\FD_{1} \cap \FD_{2}) \cap \bigcap_{j\in J} H_{j}^{-} \cap \bigcap_{j\in J} I_{0}(H_{j}^{-}), \]
and the group $\Gamma$ may be given by the generators
\begin{equation}\label{eqn:Mgenerators}
	\Gamma = \left\langle \Gamma_{1} , \quad  I_{0}^{-1} \Gamma_{1}^{\vphantom{-1}} I_{0}^{\vphantom{-1}},   \quad I_{j}^{-1}I^{\vphantom{-1}}_{0} \;\; (j\in J)  \right\rangle . 
\end{equation}
Geometrically, the conjugate copy of $\Gamma_{1}$ by $I_{0}$ corresponds to matching up two copies of the fundamental domain $\FD_{1}$ at $H_{0}$, and the elements $I_{j}^{-1}I_{0}^{\vphantom{-1}}$  correspond to the gluing isometries for the remaining sides of the fundamental domain $\FD$ corresponding to the $H_{1}$ and $\gamma_{j}(H_{j})$; equivalently the ends of $\FD_{1}$ and $\FD_{2}$ away from $H_{0}$.

\par We finally note that since the $H_{j}$ are $K$-rational hyperplanes, their corresponding reflections $I_{j}$ do indeed lie in $\PO_{f}(K)$.
\end{proof}

\begin{proof}[Proof of Theorem~\ref{thm:quasiarithsystole} (cf.\ p.\pageref{thm:quasiarithsystole})]\label{pf:mainthmproof}
	We have already seen in Prop.~\ref{prop:QuasiArith} that any lattice constructed as in \S\ref{sec:ABT} is quasi-arithmetic.
	That there are infinitely many commensurability classes of such lattices has already been demonstrated in the literature \cite[5.2]{BT:Systoles}, and follows from that fact that when these lattices are non-arithmetic, their commensurator is also a lattice in $\PO_{f}(\R)$ \cite[(B), p.298]{Margulis}: as the systole length in the construction of \S\ref{sec:ABT} decreases, one obtains a sequence of non-arithmetic lattices $\Gamma_{m}$ (for $m\in\N$), and if these were commensurable then Margulis' theorem would imply that we have some maximal lattice $\Gamma$ with $\Gamma_{m}\subset\Gamma$ for every $m$.
But since the $\Gamma_{m}$ have decreasing systole lengths, this is not possible.
\end{proof}

\section{Gromov--Piatetski-Shapiro manifolds}\label{sec:GPS}

Gromov and Piatetski-Shapiro's construction of non-arthmetic lattices in $\PO(n,1)$, from hybrid hyperbolic manifolds, is well-known and so we do not revisit it in complete detail; but we present below enough of their construction to examine the quasi-arithmeticity properties of the resulting lattices.
The reader interested in the details of the construction should find the original article accessible even to the non-expert \cite{GPS}.

\par After recalling the construction we turn to proving that these lattices are not quasi-arithmetic (\S\ref{sec:GPSnonQuasiArith}).

\subsection{The construction of Gromov and Piatetski-Shaprio}\label{sec:GPSConstruction}

In order to obtain non-arithmetic lattices $\Gamma<\PO(n,1)$ Gromov and Piatetski-Shapiro first consider two torsion-free co-compact arithmetic lattices $\tilde{\Gamma}_{1}<\PO_{f_1}(\R)$ and $\tilde{\Gamma}_{2}<\PO_{f_2}(\R)$ over a field $K$, such that the two quotient manifolds $\tilde{M}_{i}=\tilde{\Gamma}_{i}\backslash\Hyp^{n}$ ($i=1,2$) each contain a co-dimension $1$ closed submanifold $M_{0}^{(i)}$ ($i=1,2$) with an isometry $\psi\colon M_{0}^{(1)} \to M_{0}^{(2)}$.
They show that if the forms $f_{1}$ and $f_{2}$ are not similar over $K$ then the manifolds $\tilde{M}_{1}$ and $\tilde{M}_{2}$ are not commensurable: we assume that the forms are indeed not similar.
The manifold $M$ obtained by cutting each of the $\tilde{M}_{i}$ along $M_{0}^{(i)}$ and gluing the two together via the boundary isometry $\psi$, would be a cover of $M_{1}$ and $M_{2}$, if $M$ was arithmetic \cite[0.2]{GPS}.
However in light of the \emph{non}-commensurability of $M_{1}$ and $M_{2}$, the glued manifold $M$ cannot be a common cover of these spaces and so must be non-arithmetic.
Thus $\Gamma = \pi_{1}(M)$ is a non-arithmetic lattice in $\PO(n,1)$.

\begin{remark}
	The manifold $M$ is called a \emph{hybrid} manifold, obtained by a process of \emph{interbreeding}.
	For the manifolds in \S\ref{sec:ABT} we call the process \emph{inbreeding} since the two manifolds glued together arise from the same arithmetic lattice.
\end{remark}

\begin{definition}
	By a \emph{Gromov--Piatetski-Shapiro lattice} (\gps\ lattice for short) is meant a non-arithmetic lattice $\Gamma<\PO(n,1)$ obtained as $\pi_{1}(M)$ by the above procedure; that is, by taking two non-similar admissible quadratic forms $f_{1}$ and $f_{2}$ over a common field $K$, containing a common subform $f_{0}$ giving rise to an embedded codimension $1$ hypersurface.
\end{definition}

\subsection{Non-quasi-arithmeticity of GPS lattices}\label{sec:GPSnonQuasiArith}

In this section we will prove Theorem~\ref{thm:GPSnonQuasiArith} (cf.\ p.\pageref{thm:GPSnonQuasiArith}); i.e., that \gps\ lattices are not quasi-arithmetic.

\par Let $\Gamma < \PO(n,1)$ be a \gps\  lattice.
So, there is a hyperbolic manifold $M$ with $\Gamma=\pi_{1}(M)$, such that $M=M_{1}\cup_{M_0} M_{2}$ with $M_{1}$ and $M_{2}$ arising as manifolds with boundary from closed arithmetic manifolds $\tilde{M}_{i} = \Lambda_{i}\backslash\Hyp^{n}$.
As in \S\ref{sec:GPSConstruction}, suppose that $\Lambda_{i} \leq \PO_{f_i}(\Int_{K})$ (where we identify $\PO_{f_i}(\R)$ with its image in $\PO(n,1)$ via an isomorphism $\phi_{i}$ over some finite extension of $K$).
Write $\Gamma_{i}=\pi_{1}(M_{i})$ for $i=1,2$.
Then each $\Gamma_{i}$ is Zariski dense in $\PO(n,1) =  \PO_{f_n}(\R)$ \cite[0.1]{GPS}.

\par Now suppose that there is another admissible pair $(K',f')$ for $\PO(n,1)$ and such that $\Gamma \subset \PO_{f'}(K')$ (where again we identify the corresponding groups of real points by an isomorphism $\phi$ over an extension of $K'$).
Then we would have the configuration
\begin{equation}\label{eqn:GPSConfiguration}
\begin{tikzcd}%
	\phi_{1}^{-1}(\Gamma_{1}^{\vphantom{-1}}) \arrow[hook]{r}	&	{\PO_{f_1}(\Int_{K})}	\arrow[hook]{r}	&  {\PO_{f_1}(\R)}\arrow{dr}{\phi_{1}} &	& {\Gamma_{1}=\pi_{1}(M_{1})}\arrow[hook]{dl} 	\\
	{\Gamma} \arrow[hook]{r}	& 	{\PO_{f'}(K')}	 \arrow[hook]{r}	&  {\PO_{f'}(\R)}\arrow{r}{\phi}	& 	{\PO(n,1)}	& {}\\
	\phi_{2}^{-1}(\Gamma_{2}^{\vphantom{-1}}) \arrow[hook]{r}	&	{\PO_{f_2}(\Int_{K})}	\arrow[hook]{r}	&  {\PO_{f_2}(\R)}\arrow{ur}{\phi_{2}} &	& {\Gamma_{2}=\pi_{1}(M_{2})}\arrow[hook]{ul}
\end{tikzcd}.
\end{equation}

\begin{proposition}
	Assuming the above configuration, we have $K=K'$ and for each $i=1,2$, there is an  isomorphism $\PO_{f'}(K)\cong \PO_{f_i}(K)$.
\end{proposition}

\begin{proof}
By minimality of $K$ we must have $K\subseteq K'$.
Then, examining $\PO_{f'}$, we have
\[ \PO_{f'}(K' \otimes \R) \cong  \!\! \prod_{\sigma\in\Gal(K'/K)} \biggl(  \prod_{\tau\in\Gal(K/\Q)} \!\! \PO_{(f_{i}^{\sigma})^{\tau}_{\vphantom{i}}}(\R)  \biggr) ; \]
and then since $\PO_{f'}$ is admissible there can be only one non-compact factor on the right hand side of this isomorphism.
Thus $\Gal(K'/K)$ can contain only precisely one element and so $K=K'$.

\par In what follows, $\Zclosure{A}{F}$ will denote the Zariski $F$-closure of a set $A$ with respect to the Zariski $F$-topology on $\PO_{f}(F)$, for $F$ a field.
	That is, $\Zclosure{A}{F}$ is the smallest closed set in $\PO_{f}(F)$ containing $A$, in the topology generated by sets of zeroes of polynomials with coefficients in $F$.
\par Fix $i=1$ or $i=2$.
Now, $\phi^{-1}(\Gamma_{i})$ is Zariski dense in $\PO_{f'}(K)$; for if not then $\phi\bigl(\Zclosure{\phi^{-1}(\Gamma_{i})}{K}\bigr)$ would be contained in a $K$-closed proper subgroup of $\PO(n,1)$ containing $\Gamma_{i}$.
But this is also an $\R$-closed subgroup, which is impossible by Zariski density of $\Gamma_{i}$ in $\PO(n,1)$.
Similarly, $\phi^{-1}_{1}(\Gamma_{i})$ is also Zariski dense in $\PO_{f_i}(K)$.
Thus we have an isomorphism $\PO_{f_1}(K)\cong \PO_{f'}(K)$ by composing $\phi^{-1}$ and $\phi_{1}$.

\end{proof}

\noindent Thus if $\Gamma$ is a quasi-arithmetic lattice then it must be contained in $\PO_{f_i}(K)$ for both $i=1,2$, and these $\PO_{f_i}$ must be $K$-isomorphic, which is impossible since the forms $f_{i}$ are not similar \cite[2.6]{GPS}.

This concludes the proof of Theorem~\ref{thm:GPSnonQuasiArith}.{\hfill\ensuremath{\square}}

\subsection{Sub-arithmeticity}\label{sec:Sub-Arith}

In the original article describing \gps\ lattices, the authors define a \emph{sub-arithmetic} group $\Gamma<\PO(n,1)$ to be a discrete group that is Zariski dense and such that for some arithmetic lattice $\Lambda<\PO(n,1)$, we have $|\Gamma : \Lambda \cap \Gamma| < \infty$ (i.e., $\Gamma$ is virtually contained in an arithmetic lattice) \cite[0.4]{GPS}.
A sub-arithmetic group need not be a lattice and so sub-arithmeticity and quasi-arithmeticity are different notions, albeit both in the spirit of being close to arithmetic.

\par Both constructions outlined here, of non-arithmetic lattices in $\PO(n,1)$, are based on gluing manifolds with sub-arithmetic fundamental groups.
At present the only general methods for constructing non-arithmetic lattices in $\PO(n,1)$ are based on gluing constructions involving sub-arithmetic discrete groups and so the question raised by Gromov and Piatetski-Shapiro---namely, whether there exist non-arithmetic lattices that are \emph{not} constructed from sub-arithmetic subgroups---remains open \cite[0.4]{GPS}.

\providecommand{\bysame}{\leavevmode\hbox to3em{\hrulefill}\thinspace}
\providecommand{\MR}{\relax\ifhmode\unskip\space\fi MR }
\providecommand{\MRhref}[2]{%
  \href{http://www.ams.org/mathscinet-getitem?mr=#1}{#2}
}
\providecommand{\href}[2]{#2}


\begin{thebibliography}{HLMA92}

\bibitem[Ago06]{Agol:systoles}
Ian Agol, \emph{{Systoles of hyperbolic $4$-manifolds}}, preprint,
  arXiv:math/0612290v1, 2006.

\bibitem[BHC62]{Borel-Harish-Chandra}
Armand Borel and Harish-Chandra, \emph{Arithmetic subgroups of algebraic
  groups}, Ann.~of Math.~(2) \textbf{75} (1962), 485--535. \MR{0147566 (26
  \#5081)}

\bibitem[BHW11]{BHW}
Nicolas Bergeron, Fr{\'e}d{\'e}ric Haglund, and Daniel~T. Wise,
  \emph{Hyperplane sections in arithmetic hyperbolic manifolds},
  J.~Lond.~Math.~Soc.~(2) \textbf{83} (2011), no.~2, 431--448. \MR{2776645}

\bibitem[BT11]{BT:Systoles}
Mikhail~V. Belolipetsky and Scott~A. Thomson, \emph{Systoles of hyperbolic
  manifolds}, Algebraic \& Geometric Topology \textbf{11} (2011), no.~3,
  1455--1469. \MR{2821431}

\bibitem[GPS87]{GPS}
M.~Gromov and I.~Piatetski-Shapiro, \emph{Non-arithmetic groups in
  {L}obachevsky spaces}, Publications math\'ematiques de
  l'{\textsc{i.h.\'e.s.}} \textbf{66} (1987), no.~3, 93--103. \MR{0932135
  (89j:22019)}

\bibitem[HLMA92]{HLM:Borromean}
Hugh~M. Hilden, Mar{\'{\i}}a~Teresa Lozano, and Jos{\'e}~Mar{\'{\i}}a
  Montesinos-Amilibia, \emph{On the {B}orromean orbifolds: geometry and
  arithmetic}, Topology '90 ({C}olumbus, {OH}, 1990), Ohio State Univ.\ Math.\
  Res.\ Inst.\ Publ., vol.~1, de Gruyter, Berlin, 1992, pp.~133--167.
  \MR{1184408 (93h:57021)}

\bibitem[Kat07]{Katz:Systolic}
Mikhail~G. Katz, \emph{Systolic geometry and topology}, Mathematical Surveys
  and Monographs, vol. 137, American Mathematical Society, Providence, RI,
  2007, With an appendix by Jake P.~Solomon. \MR{2292367 (2008h:53063)}

\bibitem[Lee13]{Lee:SmoothMflds}
John~M. Lee, \emph{Introduction to smooth manifolds}, second ed., Graduate
  Texts in Mathematics, vol. 218, Springer, New York, 2013. \MR{2954043}

\bibitem[Mar91]{Margulis}
G.~A. Margulis, \emph{Discrete subgroups of semisimple {L}ie groups},
  Ergebnisse der Mathematik und ihrer Grenzgebiete (3) [Results in Mathematics
  and Related Areas (3)], vol.~17, Springer-Verlag, Berlin, 1991. \MR{1090825
  (92h:22021)}

\bibitem[MT62]{Mostow-Tamagawa}
G.~D. Mostow and T.~Tamagawa, \emph{On the compactness of arithmetically
  defined homogeneous spaces}, Ann.~of Math.~(2) \textbf{76} (1962), 446--463.
  \MR{0141672 (25 \#5069)}

\bibitem[Nor15]{Norfleet:ManyExamples}
Mark Norfleet, \emph{Many examples of non-cocompact {F}uchsian groups sitting
  in $\mathrm{PSL}_{2}(\mathbb{Q})$}, Geometriae Dedicata (2015), 1--9
  (English), DOI: 10.1007/s10711-015-0079-3.

\bibitem[PR94]{PlatonovRapinchuk}
Vladimir Platonov and Andrei Rapinchuk, \emph{Algebraic groups and number
  theory}, Pure and Applied Mathematics, vol. 139, Academic Press Inc., Boston,
  MA, 1994, Translated from the 1991 Russian original by Rachel Rowen.
  \MR{1278263 (95b:11039)}

\bibitem[Rat06]{Ratcliffe}
John~G. Ratcliffe, \emph{{Foundations of Hyperbolic Manifolds}}, {Second} ed.,
  Graduate Texts in Mathematics, no. 149, Springer, 2006.

\bibitem[Ser80]{Serre:Trees}
Jean-Pierre Serre, \emph{Trees}, Springer-Verlag, Berlin, 1980, Translated from
  the French by John Stillwell. \MR{607504 (82c:20083)}

\bibitem[Vin]{Vinberg:SomeExamples}
{\`E}.~B. Vinberg, \emph{Some examples of {F}uchsian groups sitting in
  $\mathrm{SL}_{2}(\mathbb{Q})$}, preprint,
  http://www.math.uni-bielefeld.de/sfb701/files/preprints/sfb12011.

\bibitem[Vin67]{Vinberg:DiscreteGroupsReflections}
\bysame, \emph{Discrete groups generated by reflections in {L}oba\v cevski\u\i
  \ spaces}, Mat.\ Sb.\ (N.S.) \textbf{72 (114)} (1967), 471--488; correction,
  ibid. 73 (115) (1967), 303. \MR{0207853 (34 \#7667)}

\bibitem[VS93]{VinbergShvartsman:Discrete}
{\`E}.~B. Vinberg and O.~V. Shvartsman, \emph{Discrete groups of motions of
  spaces of constant curvature}, Geometry, {II}, Encyclopaedia Math. Sci.,
  vol.~29, Springer, Berlin, 1993, pp.~139--248. \MR{1254933 (95b:53043)}

\bibitem[WM14]{WitteMorris}
David Witte~Morris, \emph{Introduction to arithmetic groups}, Available at:
  \texttt{\scriptsize
  http://people.uleth.ca/{\textasciitilde}dave.morris/books/IntroArithGroups.html},
  2014.

\end{thebibliography}
\end{document}